\documentclass[a4paper,11pt,titlepage,twoside]{article}

\usepackage{graphicx}
\usepackage[T1]{fontenc} 
\usepackage[utf8]{inputenc}
\usepackage[english]{babel}
\usepackage{soul}
\usepackage{amsfonts}
\usepackage{amsmath}
\usepackage{amsthm}
\usepackage{amssymb}
\usepackage{mathrsfs}
\usepackage[top=5.5cm, bottom=3cm, left=4cm, right=4cm]{geometry}
\usepackage{setspace}
\usepackage{afterpage}
\usepackage{extarrows}
\usepackage{fancyhdr}
\usepackage{titlesec}
\usepackage{enumitem} \setlist{nosep}

\theoremstyle{definition}
\newtheorem{defin}{Definition}[section]
\theoremstyle{plain}
\newtheorem{teor}[defin]{Theorem}
\newtheorem{lem}[defin]{Lemma}
\newtheorem{pro}[defin]{Proposition}
\newtheorem{cor}[defin]{Corollary}
\theoremstyle{definition}
\newtheorem{esm}[defin]{Example}
\newtheorem{osr}[defin]{Remark}
\numberwithin{equation}{section}

\numberwithin{equation}{section}

\renewcommand{\O}{\Omega}
\newcommand{\D}{\mathcal{D}}
\renewcommand{\H}{\mathcal{H}}

\newcommand{\B}{\mathcal{B}}

\newcommand{\Eo}{\mathcal{E}_\Omega}

\newcommand{\T}{\Theta}
\newcommand{\Up}{\Upsilon}
\renewcommand{\L}{\Lambda}
\newcommand{\Po}{\mathfrak{P}}
\newcommand{\n}[1]{\|#1\|}
\newcommand{\nor}{\|\cdot\|}

\newcommand{\noo}{\|\cdot\|_\O}

\renewcommand{\l}{\langle}
\renewcommand{\r}{\rangle}

\newcommand{\R}{\mathbb{R}}
\newcommand{\C}{\mathbb{C}}

\newcommand{\pint}{\l\cdot,\cdot\r}
\newcommand{\pin}[2]{\l#1 , #2\r}
\newcommand{\no}{\noindent}
\newcommand{\ol}{\overline}
\newcommand{\ull}{\underline} 

\newcommand{\Ma}{\mathscr{M}_{\ull{\alpha}}}
\newcommand{\Oa}{\O_{\ull{\alpha}}}

\newcommand{\sub}{\subseteq}

\newcommand{\mez}{\frac{1}{2}}
\renewcommand{\ll}{{\it l}}

\titleformat{\section}
{\normalfont\fillast \fontsize{12}{15}\scshape}{\thesection.}{0.8em}{}

\begin{document}

\thispagestyle{plain}

\begin{center}
	\large
	{\bf A KATO'S SECOND TYPE REPRESENTATION \\ THEOREM FOR SOLVABLE SESQUILINEAR FORMS} \\
	\vspace*{0.5cm}
	ROSARIO CORSO
\end{center}

\normalsize 
\vspace*{1cm}

\small

\begin{minipage}{11.8cm}
	{\scshape Abstract.} 
Kato's second representation theorem is generalized to solvable sesquilinear forms. These forms need not be  non-negative nor symmetric. \\
The representation considered holds for a subclass of solvable forms (called hyper-solvable), precisely for those whose domain is exactly the domain of the square root of the modulus of the associated operator. This condition always holds for closed semibounded forms, and it is also considered by several authors for symmetric sign-indefinite forms.\\
As a consequence, a one-to-one correspondence between hyper-solvable forms and operators, which generalizes those already known, is established.		
\end{minipage}

\vspace*{.5cm}

\begin{minipage}{11.8cm}
\end{minipage}

\vspace*{.5cm}

\begin{minipage}{11.8cm}
	
\end{minipage}



\normalsize

\vspace*{1cm}

		\section{Introduction}
		A sesquilinear form $\O$ on a dense domain $\D$ of a Hilbert space $\H$ is called {\it q-closed} if $\D$ can be made into a reflexive Banach space $\D[\noo]$, continuously embedded in $\H$, and such that the form is bounded in it. This allows to define a {\it Banach-Gelfand triplet} $\D \hookrightarrow \H \hookrightarrow \D^\times$, where the arrows indicate continuous embeddings and $\D^\times$ is the conjugate dual space of $\D[\noo]$. We call $\O$ {\it solvable} if a perturbation of $\O$ with a bounded form $\Up$ on $\H$, defines a bounded operator, with bounded inverse, which acts on the triplet (the set of these perturbations is denoted by $\Po(\O)$). These sesquilinear forms have been studied by Di Bella and Trapani in \cite{Tp_DB} and by Trapani and the author in \cite{RC_CT}.
		
		As proved in \cite{Tp_DB}, for a solvable sesquilinear form $\O$ there exists a closed operator $T$, with dense domain $D(T)\subseteq \D$, such that the following representation holds
		\begin{equation}
		\label{rapp_O}
		\O(\xi,\eta)=\pin{T\xi}{\eta}, \qquad \forall \xi\in D(T),\eta \in \D.
		\end{equation}
		This extends the representation theorems for sesquilinear forms considered by many authors, as for instance by Kato \cite{Kato}, McIntosh \cite{McIntosh70}, Fleige {\it et al.} \cite{FHdeS,FHdeSW',FHdeSW}, Grubi\u{s}i\'{c} {\it et al.} \cite{GKMV} and Schmitz \cite{Schmitz} (for a more complete list see the references of \cite{RC_CT}). 
		
		For a non-negative closed form $\O$, with positive associated operator $T$, Kato also proved the so-called {\it second representation theorem}  \cite[Theorem VI.2.23]{Kato}: $\D=D(T^\mez)$, where $T$ is the operator appearing in (\ref{rapp_O}), and
		$$
		\O(\xi,\eta)=\pin{T^\mez\xi}{T^\mez\eta}, \qquad \forall \xi,\eta\in \D.
		$$
		In the case where $\O$ is a general sectorial closed form, Kato \cite{Kato_2} generalized the representation as 
		$$
		\O(\xi,\eta)=\pin{T^\mez \xi}{T^{*\mez} \eta}, \qquad \forall \xi,\eta \in \D,
		$$	
		where $T^\mez$ and $T^{*\mez}$ are fractional powers of $T$ and $T^*$ (see \cite{Kato_1}), respectively, under the assumption that $\D=D(T^\mez)=D(T^{*\mez})$. However, this latter condition does not always hold, as shown by McIntosh \cite{McIntosh72}. 
		
		McIntosh \cite{McIntosh70}, Fleige {\it et al.} \cite{Fleige,FHdeS}, Grubi\u{s}i\'{c} {\it et al.} \cite{GKMV} and Schmitz \cite{Schmitz} adapted the second representation theorem for symmetric sesquilinear forms they  considered. More precisely, in \cite{Fleige,FHdeS,GKMV,Schmitz} it is proved that, if $\D=D(|T|^\mez)$ and $T=U|T|$ is the polar decomposition of $T$, then
		$$
		\O(\xi,\eta)=\pin{U|T|^\mez\xi}{|T|^\mez\eta}, \qquad \forall \xi,\eta \in \D.
		$$
		
		In this paper we adapt Kato's second representation theorem to a solvable sesquilinear form $\O$ (not necessarily symmetric), represented by an operator $T$, and with domain $\D=D(|T|^\mez)$ (if this condition is satisfied then we say that $\O$ is {\it hyper-solvable}). 
		It emerges that the condition $\D=D(|T|^\mez)=D(|T^*|^\mez)$ is equivalent to $\D=D(|T|^\mez)$ and we also prove, if $\O$ is hyper-solvable, that
		\begin{equation}
		\label{rappr_v2}
		\O(\xi,\eta)=\pin{U|T|^\mez \xi}{|T^*|^\mez \eta}=\pin{|T^*|^\mez U\xi}{|T^*|^\mez \eta}, \qquad \forall \xi,\eta\in \D
		\end{equation}
		where $T=U|T|$ is the polar decomposition of $T$.
		
		The paper is organized as follows. In Section \ref{sec:solv} we give a brief overview on q-closed and solvable forms, while in Section \ref{sec:Rad-Nik} we introduce a so-called {\it Radon-Nikodym-like representation} for a general q-closed/solvable form, i.e. an expression 
		\begin{equation}
		\label{rappRN}
		\O(\xi,\eta)=\pin{QH\xi}{H\eta}, \qquad \forall \xi,\eta \in \D,
		\end{equation}
		where $Q\in \B(\H)$ and $H$ is a positive self-adjoint operator with $0\in \rho(H)$. Moreover, we show that $\Up\in \Po(\O)$ if, and only if, $Q+H^{-1}BH^{-1}$ is a bijection of $\H$, where  $B$ is the operator associated to $\Up$. \\
		The sesquilinear forms studied in \cite{GKMV} are exactly defined as in (\ref{rappRN}) with $Q$ symmetric. Following \cite{GKMV}, we can give another expression of the operator $T$ associated to $\O$ in (\ref{rappRN}). Indeed, the domain $D(T)$ of $T$ is equal to	$D(T)=\{\xi\in \D:QH\xi\in \D\}$ and $T=HQH$.

		In Section \ref{sec:2nd} we prove the special Radon-Nikodym-like representation that holds for hyper-solvable forms
		\begin{equation*}
		\label{rappr_v1}
		\O(\xi,\eta)=\pin{V|T+B|^\mez \xi}{|T+B|^\mez \eta}, \qquad \forall \xi,\eta \in \D,
		\end{equation*}
		where $B$ is the bounded operator associated to $\Up\in \Po(\O)$ and $V\in \B(\H)$. 	If, moreover, $0\in \rho(T)$ then there exists a unique bijection $W\in \B(\H)$ such that
		\begin{equation*}
		\label{rappr1}
		\O(\xi,\eta)=\pin{W|T|^\mez \xi}{|T|^\mez \eta}, \qquad \forall \xi,\eta \in \D.
		\end{equation*}
		We also prove (\ref{rappr_v2}) and with the aid of the Radon-Nikodym-like representation we adapt the criteria, contained in Theorem 3.2 and Lemma 3.6 of \cite{GKMV}, to ensure that a solvable form is hyper-solvable. 
		
		In Section \ref{sec:corr} we consider the problem of representation to the converse direction; that is, for an operator $T$ with certain properties we construct a solvable sesquilinear form (which is in particular hyper-solvable) with associated operator $T$. More precisely, we determine a one-to-one correspondence between all hyper-solvable sesquilinear forms $\O$ and all operators $T$ such that $D(|T|^\mez)=D(|T^*|^\mez)$ and $0\in \rho(T+B)$ for some $B\in \B(\H)$. 
		This correspondence generalizes Theorem VI.2.7 of \cite{Kato} (for symmetric forms) and Theorem 5.2 of \cite{FHdeS}.

		\section{Q-closed and solvable sesquilinear forms}
		\label{sec:solv}
		
	 	\no In this paper $\H$ denotes a Hilbert space, with inner product $\pint$ and norm $\nor$, and $\D$ denotes a dense subspace of $\H$. If $\mathcal{E}$ is a Banach space, we will indicate by $\B(\mathcal{E})$ the set of all bounded operators from $\mathcal{E}$ into  itself.\\	
		
		\no Let $\O$ be a sesquilinear form defined on $\D$. The {\it adjoint} $\O^*$ of $\O$ is the form on $\D$ given by
		$$
		\O^*(\xi,\eta)=\ol{\O(\eta,\xi)} \qquad \forall \xi,\eta \in \D.
		$$
		We say that $\O$ is {\it symmetric} if $\O=\O^*$ and  {\it semi-bounded} if $\O(\xi,\xi)\geq \gamma \n{\xi}^2$ for some $\gamma \in \R$ and for all $\xi \in \D$ (in particular $\O$ is {\it non-negative} if $\gamma=0$). If there exists $M>0$ such that $|\O(\xi,\eta)|\leq M\n{\xi}\n{\eta}$ for all $\xi,\eta\in \D$ then $\O$ is said {\it bounded} (in the norm of $\H$).\\
		We denote by $\iota$ the sesquilinear form $\iota(\xi,\eta)=\pin{\xi}{\eta}$ for all $\xi,\eta\in \H$.\\

		\no We recall some definitions and properties concerning q-closed and solvable forms established in \cite{Tp_DB,RC_CT}.
				
		\begin{defin}[{\cite[Definition 5.2]{Tp_DB}}, {\cite[Proposition 3.2]{RC_CT}}]
			\label{def_q_chiusa}
			A sesquilinear form $\O$ on $\D$ is called {\it  q-closed with respect to} a norm on $\D$ which is denoted by  $\noo$ if
			\begin{enumerate}
				\item there exists $\alpha>0$ such that $\n{\xi}\leq \alpha \n{\xi}_\O$, for all $\xi \in \D$, i.e. the embedding  $\D[\noo]\to \H$ is continuous;
				\item $\Eo:=\D[\noo]$ is a reflexive Banach space;
				\item there exists $\beta >0$ such that $|\O(\xi,\eta)|\leq \beta\n{\xi}_\O\n{\eta}_\O$, for all  $\xi,\eta \in \D$, i.e. $\O$ is bounded on $\D[\noo]$.
			\end{enumerate}
			If $\Eo$ is a Hilbert space, then $\O$ is said to be {\it q-closed with respect to the inner product of $\Eo$}.
		\end{defin}

		\no Let $\O$ be a q-closed sesquilinear form with respect a norm $\noo$ on $\D$. We denote by $\Eo=\D[\noo]$ and by $\Eo^\times=\D^\times[\noo^\times]$ the conjugate dual space of $\Eo$, where $\noo^\times$ denotes the usual dual norm.\\
		We indicate the value of a conjugate linear functional $\Lambda\in \Eo^\times$ on an element $\xi \in \Eo$ by $\pin{\Lambda}{\xi}$. \\
		
		\no We denote by $\Po(\O)$ the set of bounded sesquilinear forms  $\Up$ on $\H$, such that
		\begin{enumerate}
			\item if $(\O+\Up)(\xi,\eta)=0$ for all $\eta \in \D$, then $\xi=0$;
			\item for all $\L\in \Eo^\times $ there exists $\xi\in \D$ such that 
			$$
			\pin{\L}{\eta}=(\O+\Up)(\xi,\eta), \qquad \forall \eta \in \D.
			$$
		\end{enumerate}

		\begin{defin}[{\cite[Definition 5.5]{Tp_DB}}]
			If the set $\Po(\O)$ is not empty, then $\O$ is said to be  {\it solvable with respect to}  $\noo$ (or {\it solvable with respect to the inner product} if  $\noo$ is a Hilbert norm).
		\end{defin}
	
		\begin{teor}[{\cite[Theorems 3.8, 4.4]{RC_CT}}]
			\label{th_q_cl_sol_norm_eq}	
			Let $\O$ be a q-closed (respectively solvable) sesquilinear form on $\D$ with respect to a norm $\nor_1$ and let $\nor_2$ be a norm on $\D$. Then, $\O$ is q-closed (respectively solvable) with respect to $\nor_2$ if, and only if, $\nor_1$ and $\nor_2$ are equivalent.
		\end{teor}
		
		\begin{osr}
			By the previous theorem and for simplicity of notation sometimes we will not specify a norm with respect to which a sesquilinear form is q-closed or solvable (if no ambiguity may arise).
		\end{osr}

\begin{teor}[{\cite[Theorem 4.6]{RC_CT}}]
\label{th_rapp_risol}
Let $\O$ be a solvable sesquilinear form on $\D$ with respect to a norm $\noo$. Then there exists a closed operator $T$, with dense domain $D(T)\subseteq \D$ in $\H$, such that the following statements hold.
\begin{enumerate}
	\item $\O(\xi,\eta)=\pin{T\xi}{\eta},$ for all $\xi\in D(T),\eta \in \D$.
	\item $D(T)$ is dense in $\D[\noo]$.
	\item If $\Up \in \Po(\O)$ and $B\in \B(\H)$ is the bounded operator such that
	$\Up(\xi,\eta)=\pin{B\xi}{\eta}$  for all $\xi, \eta \in \D$,
	then $0\in \rho(T+B)$. In particular, if $\Up=-\lambda \iota$, with $\lambda \in \C$, then $\lambda \in \rho(T)$, the resolvent set of  $T$.
\end{enumerate}
The operator $T$ is uniquely determined by the following condition. Let $\xi,\chi\in \H$. Then $\xi\in D(T)$ and $T\xi=\chi$ if and only if $\xi\in \D$ and $\O(\xi,\eta)=\pin{\chi}{\eta}$ for all $\eta$ belonging to a dense subset of $\D[\noo]$.
		\end{teor}

	\begin{osr}
		A densely defined and closed operator satisfying condition 1 of Theorem \ref{th_rapp_risol} needs not verify the other properties.  
		Indeed, with the notations of \cite[Sect. 1.3]{Schm}, let $\H=L^2(a,b)$, $\D=\{f\in H^1(a,b):f(a)=f(b)\}$, where $a,b\in \R$, and let $S$ be the self-adjoint operator defined by
		$S f=-i f'$ for $f\in\D$. \\
		Then the sesquilinear form $\O$ on $\D$ given by
		$$
		\O(f,g)=\pin{Sf}{Sg}=\int_a^b f'(x)\ol{g'}(x)dx, \qquad f,g\in \D
		$$
		is densely defined, non-negative and closed. In particular, $\iota\in \Po(\O)$.
		
		Let $A$ be the densely defined, closed and positive operator $Af=-f''$ for $f$ in 
		$ D(A)=H^2_0(a,b)
		$. Hence,
		$\O(f,g)=\pin{Af}{g}$ for all $f\in D(A),g\in \D$. 
		Therefore, $A$ verifies point 1 of Theorem \ref{th_rapp_risol} but not point 3, because $A$ is not self-adjoint.
	\end{osr}

	 The operator $T$ in Theorem \ref{th_rapp_risol} is called {\it associated} to $\O$.
	 The next result is the converse of statement {\it 2} of Theorem \ref{th_rapp_risol}.
	
	\begin{teor}
		\label{th_Po_ris}
		Let $\O$ be a solvable sesquilinear form on $\D$ with associated operator $T$.
		A bounded form $\Up(\cdot,\cdot)=\pin{B\cdot}{\cdot}$ belongs to $\Po(\O)$ if, and only if, $0\in \rho(T+B)$. 
		In particular,  $\Up=-\lambda \iota$, with $\lambda \in \C$, belongs to $\Po(\O)$ if and only if $\lambda \in \rho(T)$, the resolvent set of $T$.
	\end{teor}
	\begin{proof}
		We only have to show one implication.  Assume that $B\in\B(\H)$ and $0\in \rho(T+B)$. Let $\Lambda \in \Eo^\times$. Then, by hypothesis, there exist a bounded form $\Phi\in \Po(\O)$ and $\chi\in \D$ such that $\pin{\Lambda}{\eta}=(\O+\Phi)(\chi,\eta)$, for all $\eta \in \D$. Therefore, 
		$$
		\pin{\Lambda}{\eta}=(\O+\Up)(\chi,\eta)+(\Phi-\Up)(\chi,\eta), \qquad \forall \eta \in\D.
		$$
		Since $\Phi-\Up$ is a bounded form on $\H$ and $0\in \rho(T+B)$, there exists $\xi\in D(T)$ such that $(\Phi-\Up)(\chi,\eta)=\pin{(T+B)\xi}{\eta}=(\O+\Up)(\xi,\eta)$, for all $\eta \in \D$. In conclusion we have,
		$$
		\pin{\Lambda}{\eta}=(\O+\Up)(\chi+\xi,\eta), \qquad \forall \eta \in\D,
		$$
		hence $\Up\in \Po(\O)$.
	\end{proof}


		\begin{cor}
			Let $\O$ be a symmetric solvable sesquilinear form. Then $-\lambda\iota\in \Po(\O)$ for all $\lambda \in  \C\backslash \R$.
		\end{cor}
	\begin{proof}
		It is a direct consequence of \cite[Corollary 4.14]{RC_CT} and Theorem \ref{th_Po_ris}.
	\end{proof}

		Kato proved a representation theorem for so-called {\it closed sectorial} (densely defined) sesquilinear forms (\cite[Ch. VI]{Kato}). 
		In the symmetric case, these forms (i.e. semi-bounded forms) coincide exactly with symmetric solvable forms.
	
		\begin{pro}
			\label{pro_semi_solv->clos}
			Let $\O$ be a densely defined semi-bounded form. Then, $\O$ is closed in Kato's sense if, and only if, it is solvable.
		\end{pro}
		\begin{proof}
			Let $\gamma \in \R$ such that $\O(\xi,\xi)\geq \gamma \n{\xi}^2$, for all $\xi\in \D$. 
			Suppose that $\O$ is solvable. Then, by \cite[Corollary 4.14]{RC_CT}, the operator $T$ associated to $\O$ is self-adjoint. This implies that there exists $\lambda \in \rho(T)\cap \R$ such that $\lambda<\gamma$, and $-\lambda \iota \in \Po$ by Theorem \ref{th_Po_ris}. Hence, the statement is proved applying \cite[Proposition 7.1]{RC_CT}.
		\end{proof}

		\section{Radon-Nikodym-like representation theorem}
		\label{sec:Rad-Nik}
		
		In the next section we will introduce the Kato's second type representation for solvable sesquilinear forms (Theorem \ref{2_repr_th_2}). A crucial hypothesis of this theorem is the condition that the domain of a solvable sesquilinear form coincides exactly with the domain of the square root of the modulus of the associated operator. 
		
		Differently, in the present section we give a representation of general q-closed and solvable sesquilinear forms with respect to an inner product. \\
		We start recalling a lemma which derives from the Heinz inequality and an application of it.
		
		\begin{lem}[{\cite[Corollary 1.3]{Curgus}}]
			\label{P1,2^alfa}
			If $P_1$ and $P_2$ are positive self-adjoint operators on $\H$ and $D(P_1)=D(P_2)$, then, for all $0\leq \alpha \leq 1$,
			$$
			D(P_1^\alpha)=D(P_2^\alpha),
			$$
			and the corresponding graph norms on $D(P_1^\alpha)=D(P_2^\alpha)$ are equivalent.
		\end{lem}
		
		\begin{cor}
			\label{cor_sqrt}
			Let $T$ be a closed and densely defined operator on $\H$ and $B\in \B(\H)$. Then $D(|T|^\mez)=D(|T+B|^\mez)$ and the graph norms of $|T|^\mez$ and $|T+B|^\mez$ are equivalent. Moreover, if $0\in \rho(T+B)$, then these norms are also equivalent to the norm defined by
			$$
			\xi \mapsto \n{|T+B|^\mez \xi}, \qquad  \xi \in D(|T|^\mez).
			$$
		\end{cor}
		
		\begin{teor}
			\label{th_H1_H2}
			Let $H_1$, $H_2$ be two positive self-adjoint operators with the same domain $\D$ and such that $0\in \rho(H_1)\cap \rho(H_2)$. Let $Q\in \B(\H)$ and consider the sesquilinear form 
			\begin{equation}
			\label{form_QH1H2}
			\O(\xi,\eta)=\pin{QH_1\xi}{H_2\eta}, \qquad  \xi,\eta\in \D.
			\end{equation}
			Then, $\O$ is q-closed (with respect to an inner product).\\
			Moreover, if $Q$ is an isomorphism of $\H$, then $\O$ is solvable, $0\in \Po(\O)$ and its associated operator is $T=H_2QH_1$
			defined in the natural domain $D(T)=\{\xi\in \D: QH_1\xi \in \D\}$.
		\end{teor}
		\begin{proof}
			By the closed graph theorem the norms given by
			$$
			\n{\xi}'_\O=\n{H_1\xi} \;\text{ and }\;
			\n{\xi}_\O=\n{H_2\xi} \qquad  \xi \in \D
			$$
			are equivalent.
			Hence, $\Eo:=\D[\noo]$ is a Hilbert space and there exists $\alpha >0$ such that $\n{\xi}\leq \alpha \n{\xi}_\O$, for all $\xi\in \D$.		
			Moreover, we have, for all $\xi,\eta \in \D$,
			\begin{eqnarray*}
			|\O(\xi,\eta)|&=&|\pin{QH_1\xi}{H_2\eta}| \\
			&\leq& \n{Q}\n{H_1\xi}\n{H_2 \eta} \\
			&=& \n{Q}\n{\xi}'_\O\n{\eta}_\O \\
			&\leq& \beta\n{Q}\n{\xi}_\O\n{\eta}_\O
			\end{eqnarray*}
			for some constant $\beta>0$. Therefore, $\O$ is q-closed with respect to the norm $\noo$. Now assume that $Q$ is an isomorphism. We will prove that  $\Up\in \Po(\O)$, where $\Up=0$. Denote with $\pint_\O$ the inner product which induces $\noo$. If $\O(\xi,\eta)=0$, for all $\eta \in \D$, i.e.
			$$
			\pin{QH_1\xi}{H_2\eta}=0, \qquad \forall \eta\in \D,
			$$
			then $QH_1\xi=0$ since $0\in \rho(H_2)$, and hence $\xi=0$ by the invertibility of $QH_1$. \\
			Let $\Lambda\in \Eo^\times$, then by Riesz's Lemma there exists $\chi \in \D$ such that $\pin{\Lambda}{\eta}=\pin{\chi}{\eta}_\O$, for all $\eta\in \D$, i.e.
			$$
			\pin{\Lambda}{\eta}=\pin{H_2\chi}{H_2\eta}, \qquad \forall \eta \in \D.
			$$
			Since $QH_1$ is invertible, there exists $\xi\in\D$ such that $QH_1 \xi =H_2 \chi$, so we finally have
			$$
			\O(\xi,\eta)=\pin{QH_1\xi}{H_2\eta}=\pin{H_2\chi}{H_2\eta}=\pin{\Lambda}{\eta}, \qquad \forall \eta \in \D.
			$$
			This proves that $\O$ is solvable. Put $T'=H_2QH_1$
			on the natural domain $D(T')=\{\xi\in \H: QH_1\xi \in \D\}$. Then $0\in \rho(T')$ and
			$$
			\O(\xi,\eta)=\pin{QH_1 \xi}{H_2 \eta}=\pin{T' \xi}{\eta}, \qquad \forall \xi\in D(T'),\eta\in \D.
			$$
			Therefore, by Theorem \ref{th_rapp_risol}, $T'$ is a restriction of the operator $T$ associated to $\O$ and $0\in\rho(T)$. Hence, $T'=T$.
		\end{proof}
		
		\begin{osr}
			Lemma 3.1 generalizes Theorem 2.3 of \cite{GKMV} in which the authors consider $H_1=H_2$ and $Q$ symmetric. 					
		\end{osr}
				
		\begin{cor}
			Let $H_1$, $H_2$ be two positive self-adjoint operators with the same domain $\D$ and such that $0\in \rho(H_1)\cap \rho(H_2)$. Let $Q,B\in \B(\H)$ be such that $Q+H_2^{-1}BH_1^{-1}$ is an isomorphism of $\H$. Then, the sesquilinear form 
			$$
			\O(\xi,\eta)=\pin{QH_1\xi}{H_2\eta}, \qquad \xi,\eta\in \D,
			$$ 
			is solvable  (with respect to an inner product) and its associated operator is 
			$$
			T=H_2QH_1
			$$
			defined in the natural domain $D(T)=\{\xi\in \H: QH_1\xi \in \D\}$.
		\end{cor}
		\begin{proof}
			Setting $\Up(\xi,\eta)=\pin{B\xi}{\eta}$, for all $\xi,\eta\in \D$, the sesquilinear form 
			$$
			(\O+\Up)(\xi,\eta)=\pin{(Q+H_2^{-1}BH_1^{-1})H_1 \xi}{H_2 \eta}, \qquad \forall \xi,\eta \in \D,
			$$
			is solvable by Theorem \ref{th_H1_H2}. Hence, $\O$ is solvable, $\Up\in \Po(\O)$, and  the operator associated to $\O+\Up$ is
			$$
			T'=H_2(Q+H_2^{-1}BH_1^{-1})H_1.
			$$
			Taking into account that $H_2^{-1}BH_1^{-1}:\H\to \D$ we have 
			$$
			D(T')=\{\xi\in \H: (Q+H_2^{-1}BH_1^{-1})H_1\xi \in \D\}=\{\xi\in \H: QH_1\xi \in \D\}
			$$
			and
			$$
			T'=H_2(Q+H_2^{-1}BH_1^{-1})H_1=H_2QH_1+B.
			$$
			Therefore, $T=H_2QH_1$ is the operator associated to $\O$.
		\end{proof}
		
	\begin{osr}
		A special case of the previous corollary occurs if there exists $\lambda\in \C$, such that $Q-\lambda H_2^{-1}H_1^{-1}$ (or  $Q-\lambda H^{-2}$ if, also, $H_1=H_2=:H$) is an isomorphism of $\H$. In this case $-\lambda \iota\in \Po(\O)$.
	\end{osr}

		Now we prove this converse lemma.

		\begin{lem}
			\label{Rad-Nik}
			Let $\O$ be a q-closed sesquilinear form on $\D$ with respect to an inner product. Then one has the following.
			\begin{enumerate}
				\item There exists a positive self-adjoint operator $H$ with domain $D(H)=\D$ and $0\in \rho(H)$.
				\item Let $H$ be a positive self-adjoint operator with domain $D(H)=\D$ and $0\in \rho(H)$.  Then there exists a unique $Q\in \B(\H)$ such that 
				\begin{equation}
				\label{rapp_Rad_Nik}
				\O(\xi,\eta)=\pin{QH\xi}{H\eta}, \qquad \forall \xi,\eta \in \D.
				\end{equation}
				If, in addition, $\O$ is solvable, $\Up\in \Po(\O)$ and $B\in \B(\H)$ is the bounded operator such that $\Up(\xi,\eta)=\pin{B\xi}{\eta}$ for all $\xi,\eta\in \H$, then $Q+H^{-1}BH^{-1}$ is an isomorphism of $\H$.
			\end{enumerate}
			
		\end{lem}
		\begin{proof}
			\begin{enumerate}
			\item Let $\pint_*$ be an inner product with respect to which $\O$ is q-closed. Then,  $\pint_*$ is a non-negative closed sesquilinear form such that $\alpha\pin{\xi}{\xi}\leq \pin{\xi}{\xi}_*$ for some $\alpha>0$ and for all $\xi\in \D$. By Kato's second representation theorem \cite[Theorem VI.2.23]{Kato}, there exists a positive self-adjoint operator $H$ with domain $\D$ and resolvent set containing $0$. 		
			\item Let $H$ be a positive self-adjoint operator with domain $D(H)=\D$ and $0\in \rho(H)$. Then, by the closed graph theorem, the norm 
			induced by the inner product
			$$
			\pin{\xi}{\eta}_\O=\pin{H\xi}{H\eta}, \qquad \forall \xi,\eta \in \D
			$$ 
			is equivalent to the one induced by $\pint_*$.  Therefore, by Theorem \ref{th_q_cl_sol_norm_eq} $\O$ is q-closed with respect to  $\pint_\O$.\\
			By the boundedness of $\O$ in $\D[\pint_\O]$, we have, for some $\beta>0$, 
			$$
			|\O(\xi,\xi)|\leq \beta\n{H\xi}^2, \qquad \forall \xi \in \D,
			$$
			and using \cite[Lemma VI.3.1]{Kato}, there exists $Q\in \B(\H)$ such that 
			$$
			\O(\xi,\eta)=\pin{QH\xi}{H\eta}, \qquad \forall \xi,\eta \in \D.
			$$
			The uniqueness of $Q$ follows from the fact that $0\in \rho(H)$.\\
			Assume now that $\O$ is also solvable and $\Up\in\Po(\O)$. Let $B$ be the operator associated to $\Up$. We have 
			$$
			(\O+\Up)(\xi,\eta)=\pin{(Q+H^{-1}BH^{-1})H\xi}{H\eta}, \qquad \forall \xi,\eta \in \D,
			$$
			By the definition of solvability, $Q+H^{-1}BH^{-1}$ is an isomorphism of $\H$. In particular, to prove that $Q+H^{-1}BH^{-1}$ is surjective, let $\chi\in \H$. Then the functional $\pin{\Lambda}{\eta}=\pin{\chi}{H\eta}$ is bounded on $\Eo$, because $|\pin{\Lambda}{\eta}|\leq \n{\chi}\n{\eta}_\O$, by Cauchy-Schwarz inequality. Then, there exists $\xi\in \D$ such that 
			$ \pin{\Lambda}{\eta} =(\O+\Up)(\xi,\eta)$, for all $\eta \in \D$, which imply $(Q+H^{-1}BH^{-1})H\xi=\chi$.
		\end{enumerate}
		\end{proof}
		
		Hence, one can formulate the following characterization.
		
		\begin{teor}
			\label{car_Rad_Nik}
			A sesquilinear form $\O$ on $\D$ is q-closed with respect to an inner product if and only if there exist a positive self-adjoint operator $H$, with domain $D(H)=\D$ and $0\in \rho(H)$, and $Q\in \B(\H)$ such that 
			\begin{equation}
			\label{Rad_Nik_like}
			\O(\xi,\eta)=\pin{QH\xi}{H\eta}, \qquad \forall \xi,\eta \in \D.
			\end{equation}
			Suppose that \emph{(\ref{Rad_Nik_like})} holds with $Q$ and $H$ as above. Then
					\begin{enumerate}
						\item a bounded form $\Up$ with associated operator $B$ belongs to $\Po(\O)$ if and only if $Q+H^{-1}BH^{-1}$ is a bijection of $\H$;
						\item if $\O$ is also solvable then its associated operator is $T=HQH$
						defined on the natural domain $D(T)=\{\xi\in \D: QH\xi \in \D\}$.
					\end{enumerate}
		\end{teor}
		
		\begin{defin}
			Let $\O$ be a q-closed sesquilinear form on $\D$ with respect to an inner product. An expression like (\ref{Rad_Nik_like}) (with $Q,H$ as in the statement) is called a {\it Radon-Nikodym-like representation} of $\O$.
		\end{defin}
		
		\begin{osr}
			\label{rem_T-reg}
			Assume that $\O$ is a q-closed sesquilinear form with respect to an inner product $\pint_\O$. Then $\O$ is $\pint$-regular in the sense of \cite{Tp_DB} (in particular, $\pint_\O$ is $\pint$-absolutely continuous by \cite[Corollary 2.3]{RC_CT}) and therefore it admits a Radon-Nikodym representation according to \cite[Theorem 3.6]{Tp_DB} as follows
			\begin{equation}
			\label{pas_Rad_Nik_orig}
				\O(\xi,\eta)=\pin{KY\xi}{K\eta}, \qquad\forall \xi,\eta \in \D,
			\end{equation}
			where $K$ is a positive, self-adjoint operator which domain $D(K)$ contains $\D$ and $Y$ is an operator from $\D$ into $D(K)$. Moreover, $0\in \rho(K)$, as one can see in the proof of \cite[Theorem 3.6]{Tp_DB}. According to \cite[Remark 3.7]{Tp_DB}, $\O$ admits a representation 
			\begin{equation}
				\label{pas_Rad_Nik_orig2}
				\O(\xi,\eta)=\pin{SK\xi}{K\eta}, \qquad\forall \xi,\eta \in \D,
			\end{equation}
			where $S\in \B(\H)$. Hence, if $D(K)=\D$, then (\ref{pas_Rad_Nik_orig2}) is exactly a representation (\ref{Rad_Nik_like}).
			This motivates the terminology "Radon-Nikodym-like" representation. Actually, if $H$ is as in Lemma \ref{Rad-Nik}, then $\pint_\O:=\pin{H\cdot}{H\cdot}$ defines an inner product with respect to which $\O$ is q-closed, and following Section 6 of \cite{RC_CT}, there exists an operator $ P\in \B(\Eo)$, where $\Eo:=\D[\pint_\O]$, such that
			$$
			\O(\xi,\eta)=\pin{HP\xi}{H\eta},\qquad \forall\xi,\eta \in \D,
			$$	
			which is an expression like (\ref{pas_Rad_Nik_orig}).
			
			Radon-Nikodym type theorems for sesquilinear forms were previously given in the non-negative case by Sebestyén and Titkos in \cite{Seb_Tit} and by Tarcsay \cite{Tarc}.
		\end{osr}
		
		\begin{osr}
			Theorems \ref{th_q_cl_sol_norm_eq} and \ref{car_Rad_Nik}  imply that a necessary and sufficient condition for a solvable sesquilinear form $\O$ to have a Radon-Nikodym representation is that $\O$ is solvable with respect to a norm equivalent to one induced by an inner product. 
		\end{osr}
		
		We conclude this section with a simple relation between a sesquilinear form and its adjoint which gives also another proof of Theorem 4.11 of \cite{RC_CT} (in the case of q-closed/solvable forms with respect to an inner product).
		
		\begin{pro}
			\label{pro_RN_agg}
			Let $\O$ be a q-closed sesquilinear form on $\D$ with a Radon-Nikodym-like representation
			$$
			\O(\xi,\eta)=\pin{QH\xi}{H\eta}, \qquad \forall \xi,\eta \in \D.
			$$
			Then, $\O^*$ has a Radon-Nikodym-like representation
			$$
			\O^*(\xi,\eta)=\pin{Q^*H\xi}{H\eta}, \qquad \forall \xi,\eta \in \D.
			$$
			If in addition $\O$ and $\O^*$ are solvable, the operators $HQH$ and $HQ^*H$ are those associated to $\O$ and $\O^*$, respectively.
		\end{pro}

		\section{The second representation theorem}
		\label{sec:2nd}
		
		The second representation theorem holds for hyper-solvable sesquilinear forms, in the sense of the next definition.
				
		\begin{defin}
			A solvable sesquilinear form on $\D$ with associated operator $T$ is said {\it hyper-solvable} if $\D=D(|T|^\mez)$.
		\end{defin}
		
		\begin{osr}
			Fleige {\it et al.} in \cite{Fleige,FHdeS,FHdeSW',FHdeSW} use the name "regular" instead of "hyper-solvable". We use a different terminology because we do not want to confuse hyper-solvable forms with $\T$-regular forms (see Remark \ref{rem_T-reg}).
		\end{osr}
	
		The next result is an application of the closed graph theorem.
		
		\begin{pro}
			\label{pro_equiv_norm}
			If $\O$ is a hyper-solvable form on $\D$ with respect to a norm $\noo$, then $\noo$ is equivalent to the graph norm of $|T|^\mez$, where $T$ is the operator associated to $\O$.
		\end{pro}

		\begin{cor}
			\label{cor_hyper_inner}
			Every hyper-solvable sesquilinear form is solvable with respect to an inner product.
		\end{cor}
		
		\begin{esm}[{\cite[Example 6.1]{Tp_DB}}, {\cite[Example 5.5]{RC_CT}}]
			\label{esm_simpl_1}
			Let $\{\alpha_n\}$ be a sequence of complex numbers and let $\Oa$ be the sesquilinear form on 
			$$
			\D=\left \{\{\xi_n\}\in \ll_2:\sum_{n=1}^\infty |\alpha_n| |\xi_n|^2< \infty\right \}
			$$
			given by
			$$
			\Oa(\{\xi_n\},\{\eta_n\})=\sum_{n=1}^{\infty} \alpha_n\xi_n\ol{\eta_n}, \qquad  \{\xi_n\},\{\eta_n\}\in \D.
			$$
			The sesquilinear form $\Oa$ is solvable. In particular, consider 
			the sequence $\ull{\beta}=\{\beta_n\}$ such that $\beta_n=-\alpha_n+2$ if $|\alpha_n|\leq 1$, and $\beta_n=0$ if $|\alpha_n|> 1$. Then, by $0\notin \ol{\{\alpha_n+\beta_n\}}$, it is easily to see that the bounded sesquilinear form
			$$
			\O_\beta(\{\xi_n\},\{\eta_n\})=\sum_{n=1}^{\infty} \beta_n\xi_n\ol{\eta_n}, \qquad  \{\xi_n\},\{\eta_n\}\in\H
			$$
			belongs to $\Po(\Oa)$. The operator $\Ma$ associated to $\Oa$ is defined on the domain 
			$$
			D(\Ma)=\left \{\{\xi_n\}\in \ll_2:\sum_{n=1}^\infty |\alpha_n \xi_n|^2< \infty\right \}
			$$
			and acts as follows
			$\Ma \{\xi_n\}=\{\alpha_n\xi_n\}$ for all $\{\xi_n\} \in D(\Ma).$
			Clearly, $D(|\Ma|^\mez)=\D$; hence $\Oa$ is hyper-solvable.
		\end{esm}

		\begin{esm}
			\label{esm_simpl_2}
			Let $\mathcal{L}$ be the Lebesgue measure on $\C$, $r:\C\to \C$ be a measurable function and $\O$ the sesquilinear form with domain
			$$
			\D:=\left \{f\in L^2(\C): \int_{\C}|r(z)||f(z)|^2d\mathcal{L}< \infty \right\}
			$$
			and given by
			$	\displaystyle
			\O(f,g)=\int_\C r(z)f(z)\ol{g(z)}d\mathcal{L},$ for all $f,g\in \D$.\\
			It is easy to see that $\O$ is q-closed with respect to the norm
			$$ \n{f}_\O=\left (\int_{\C}(1+|r(z)|)|f(z)|^2d\mathcal{L}\right )^{\frac{1}{2}}, \qquad  f\in \D.$$
			Let $Z:=\{z\in \C: |r(z)|\leq 1\}$, $B$ be the bounded operator given by
			\begin{equation}
			\label{B_esm_simpl_2}
				(Bf)(z)=(1-r(z))\chi_Z(z)f(z), \qquad  f \in L^2(\C),
			\end{equation}
			where $\chi_Z$ is the characteristic function on $Z$, 
			and let 
			\begin{equation}
			\label{Up_esm_simpl_2}
				\Up(f,g)=\int_\C (1-r(z))\chi_Z(z)f(z)\ol{g(z)}d\mathcal{L}, \qquad  f,g\in L^2(\C).
			\end{equation}
			Following Example 7.3 of \cite{RC_CT}, one can prove that $\Up\in \Po(\O)$, i.e. $\O$ is solvable with respect to the norm $ \n{\cdot}_\O$.\\
			Now, we show that the operator $T$ associated to $\O$ is the multiplication operator $M$ by $r$, $(M f)(z)=r(z)f(z)$, with domain
			$$
			D(M)=\left \{f\in L^2(\C): \int_{\C} |r(z)f(z)|^2d\mathcal{L} < \infty \right \}\subset \D.
			$$ 
			Indeed, we have 
			$$
			\O(f,g)=\int_{\C} (M f)(z)\ol{g(z)}d\mathcal{L}, \qquad \forall f\in D(M),g\in \D,
			$$	
			and $0\in \rho(M +B)$; therefore, by Theorem \ref{th_rapp_risol}, $T=M$. Since $D(|M|^\mez)=\D$, $\O$ is hyper-solvable.	
		\end{esm}

	\begin{osr}
		Examples of (symmetric) solvable sesquilinear forms which are not hyper-solvable can be found in Example 5.4 of \cite{FHdeS}, in Example 2.11, in Remark 3.7 and in Proposition 4.2 of \cite{GKMV}.
	\end{osr}

	Hyper-solvable sesquilinear forms have a special Radon-Nikodym-like representation.
				
		\begin{teor}
			\label{2_repr_th_1}
			Let $\O$ be a solvable sesquilinear form on $\D$ and $T$ be the operator associated to $\O$. Suppose, moreover, that $\Up\in \Po(\O)$ and $B$ is the operator associated to $\Up$. \\
			If $\O$ is hyper-solvable, then there exists a unique operator $V\in \B(\H)$ such that 
			$$
			\O(\xi,\eta)=\pin{V|T+B|^\mez\xi}{|T+B|^\mez\eta}, \qquad \forall \xi,\eta \in \D.
			$$
			In particular, if  $0\in \rho(T)$, then there exists a unique bijection $W\in \B(\H)$ such that
			\begin{equation}
			\label{eq_rapp_WTT}
			\O(\xi,\eta)=\pin{W|T|^\mez\xi}{|T|^\mez\eta}, \qquad \forall \xi,\eta \in \D.
			\end{equation}
		\end{teor}
		\begin{proof}
			We have $D(|T+B|^\mez)=D(|T|^\mez)=\D,$	by Lemma \ref{P1,2^alfa}. 
			Theorem \ref{th_rapp_risol} implies that $0\in \rho(T+B)$ and therefore $0\in \rho(|T+B|^\mez)$. So, the statement is proved by applying Lemma \ref{Rad-Nik}. \\
			If $0\in \rho(T)$ then one can choose $B=0$ in Theorem \ref{th_Po_ris}.
		\end{proof}

		Now, basing on a Radon-Nikodym-like representation,  we establish necessary and sufficient conditions for a solvable sesquilinear form to be hyper-solvable (Theorem \ref{cns_D=}). These conditions are inspired by those presented in Theorem 3.2 of \cite{GKMV}. More precisely, the following two lemmas hold.

		\begin{lem}
			\label{cns_D<}
			Let $\O$ be a solvable sesquilinear form on $\D$ with respect to an inner product and with associated operator $T$. Let 
			$$
			\O(\xi,\eta)=\pin{QH\xi}{H\eta}, \qquad \forall \xi,\eta\in \D,
			$$			
			be a Radon-Nikodym-like representation of $\O$. Suppose, moreover, that $\Up\in \Po(\O)$, $B$ is the operator associated to $\Up$, and that $T+B=U_B|T+B|$ is the polar decomposition of $T+B$. The following statements are equivalent.
			\begin{enumerate}
				\item $D(|T|^\mez)\supseteq\D$;	
				\item the operator $X=H^{-1}|T+B|H^{-1}$ is bounded on $D(X):=HD(T)$;
				\item the operator $K_1=HU_BH^{-1}$ is bounded on $\H$;
				\item $U_B\D\subseteq\D$.				
			\end{enumerate}
		\end{lem}
		\begin{proof}
			Before to prove the equivalences we note some facts. First of all, $0\in \rho(T+B)\cap \rho(T^*+B^*)$ by Theorem \ref{th_rapp_risol} and, hence, $U_B$ is unitary. An inner product with respect to which $\O$ is solvable can be defined as $\pint_\O=\pin{H\cdot}{H\cdot}$. Moreover, $D(X)$ is dense in $\H$. Indeed, assume that $\chi\in \H$ and $\pin{H\xi}{\chi}=0$, for all $\xi \in D(T)$.
			Then, there exists $\eta \in \D$ such that $H\eta=\chi$ and hence
			$$
			0=\pin{H\xi}{\chi}=\pin{H\xi}{H\eta}=\pin{\xi}{\eta}_\O, \qquad \forall \xi\in D(T).
			$$
			But $D(T)$ is the domain of the operator $T$ associated to $\O$, hence it is dense in $\D[\pint_\O]$. It follows that $\eta=0$ and $\chi=0$.
			
			{\it 1. $\Leftrightarrow$ 2.} 
			The operator $|T+B|^\mez H^{-1}$ is densely defined on its natural domain 
			$$
			D':=\{\xi\in \H:H^{-1}\xi \in D(|T|^\mez)\},
			$$ 
			because $D'\supseteq D(X)=HD(T)$. Moreover, $|T+B|^\mez H^{-1}$ is closed.\\
			We have 
			$$
			(|T+B|^\mez H^{-1})^*(|T+B|^\mez H^{-1})\supseteq H^{-1} |T+B|^\mez |T+B|^\mez H^{-1}=X.
			$$
			Taking into account that $D(H)=\D$,
			by the closed graph theorem, $D(|T|^\mez)\supseteq\D$ implies that $|T+B|^\mez H^{-1}$ is bounded and in particular that $X$ is bounded. Conversely, suppose that $X$ is bounded. Therefore,
			$$
			||T+B|^\mez H^{-1}|^2=(|T+B|^\mez H^{-1})^*(|T+B|^\mez H^{-1})=\ol{X},
			$$
			the closure of $X$, which is a bounded operator on whole $\H$. Consequently, $|T+B|^\mez H^{-1}$ is defined on $D'=\H$; i.e., $D(|T|^\mez)\supseteq\D$.  
			
			{\it 2. $\Leftrightarrow$ 3.} 
			The operator $K_1$ is closed on the natural domain 
			$D(K_1)=\{\xi\in \H: U_B H^{-1}\xi \in \D \}.$
			We have $D(X)\sub D(K_1)$. Indeed, if $\xi \in D(X)=HD(T)$ then $\xi=H\eta$ with $\eta \in D(T)$ and, taking into account that $U_BD(T)=D(T^*)\cap R(U_B)\sub \D$ (see \cite[Sec. 7.1]{Schm}), $U_BH^{-1}\xi=U_B\eta \in \D$. Hence, $\xi\in D(K_1)$.\\
			Let $Q_B:=Q+H^{-1}BH^{-1}$. Then $Q_B^*$ is invertible with bounded inverse by Theorem \ref{car_Rad_Nik} and $H{Q_B^*}H=T^*+B^*$ by Proposition \ref{pro_RN_agg} and Theorem 4.11 of \cite{RC_CT}. If $\xi\in D(X)$ we get
			\begin{eqnarray*}
				{Q_B^*}^{-1}X\xi &=& {Q_B^*}^{-1}H^{-1} |T+B| H^{-1} \xi \\
				&=& HH^{-1}{Q_B^*}^{-1}H^{-1}|T+B| H^{-1} \xi  \\
				&=& H(T^*+B^*)^{-1}|T+B| H^{-1}\xi \\
				&=& HU_B H^{-1} \xi \\
				&=& K_1 \xi.
			\end{eqnarray*}
		    Clearly, since $Q_B^*\in \B(\H)$, if $K_1$ is bounded then $X$ is bounded too. Conversely, assume that $X$ is bounded. Since $K_1$ is closed and $D(X)\sub D(K_1)$ is dense (see above), $K_1$ is a bounded operator on $\H$.
		    
			{\it 3. $\Rightarrow$ 4.} It follows easily by the definition of  $D(K_1)$ and by $D(H)=\D$.
			
			{\it 4. $\Rightarrow$ 3.} It is a consequence of the closed graph theorem.
		\end{proof}

		\begin{lem}
			\label{cns_D>}
			Assume the same hypotheses of Lemma \ref{cns_D<}. 
			The following statements are equivalent.
			\begin{enumerate}
				\item $D(|T|^\mez)\subseteq\D$;	
				\item the operator $Y=H|T+B|^{-1}H$ on $\D$ is bounded;
				\item the operator $K_2=HU_B^*H^{-1}$ is bounded on $\H$;
				\item $U_B\D\supseteq\D$.				
			\end{enumerate}
		\end{lem}
		\begin{proof}
		We sketch the proof which is similar to the previous one.
		
		{\it 1. $\Leftrightarrow$ 2.} 
		The operator $H|T+B|^{-\mez}$ on the natural domain is closed, but also densely defined, since	$D(|T|^\mez)\subseteq D(H|T+B|^{-\mez})$. 
		Moreover, 
		$$
		H|T+B|^{-\mez}(H|T+B|^{-\mez})^*\supseteq Y. 
		$$
		Then, $Y$ is bounded if, and only if,  $D(|T|^\mez)\subseteq\D$.
		
		{\it 2. $\Leftrightarrow$ 3.} 
		The operator $K_2=HU_B^*H^{-1}$ is defined on 
		$D(K_2)=\{\xi \in \H: U_B^*H^{-1}\xi \in \D\}$
		and it is closed.
		It results that ${Q_B^*}^{-1}\D\sub D(K_2)$. Indeed, if $\xi\in {Q_B^*}^{-1}\D$ then $\xi={Q_B^*}^{-1}H^{-1}\chi$, for some $\chi \in \H$, and 
		$$
		U_B^*H^{-1}\xi=U_B^* H^{-1}{Q_B^*}^{-1}H^{-1} \chi =U_B^* (T^*+B^*)^{-1} \chi=U_B^* \eta
		$$
		with some $\eta \in D(T^*)$. 
		Therefore, $\xi\in D(K_2)$ since $U_B^*H^{-1}\xi=U_B^* \eta\in D(T)$. \\
		Moreover, if $\xi\in D(YQ_B^*)={Q_B^*}^{-1}\D$, then $YQ_B^*\xi = K_2 \xi$.\\
		Taking into account that $D(YQ_B^*)$ is dense in $\H$ ($Q_B^*$ is an isomorphism of $\H$) $Y$ is bounded if and only if $K_2$ is bounded.
		
		{\it 3. $\Leftrightarrow$ 4.} By definition of $D(K_2)$ and since $U_B$ is unitary, both statements are equivalent to $U_B^*\D\subseteq\D$.		
		\end{proof}
		
		Combining Lemmas \ref{cns_D<} and \ref{cns_D>} we obtain the following characterization.
		
		\begin{teor}
			\label{cns_D=}
			Let $\O$ be a solvable sesquilinear form on $\D$ with respect to an inner product  and with associated operator $T$. Let 
			$$
			\O(\xi,\eta)=\pin{QH\xi}{H\eta}, \qquad \forall \xi,\eta\in \D,
			$$			
			be a Radon-Nikodym-like representation of $\O$. Suppose, moreover, that $\Up\in \Po(\O)$, $B$ is the operator associated to $\Up$, and that $T+B=U_B|T+B|$ is the polar decomposition of $T+B$. The following statements are equivalent.
			\begin{enumerate}
				\item $\O$ is hyper-solvable, i.e. $D(|T|^\mez)=\D$;	
				\item the operators $X=H^{-1}|T+B|H^{-1}$ and $Y=H|T+B|^{-1}H$ defined on $D(X)=HD(T)$ and $\D$, respectively, are bounded;
				\item the operators $K_1=HU_BH^{-1}$ and $K_2=HU_B^*H^{-1}$ are bounded on $\H$;
				\item $U_B\D=\D$.				
			\end{enumerate}
		\end{teor}

		\begin{esm}
			Let $\O$ be the solvable sesquilinear form of Example \ref{esm_simpl_2} defined on the domain 
			$\D:=\left \{f\in L^2(\C): \int_{\C}|r(z)||f(z)|^2dz< \infty \right\}$
			and $M$ be its associated operator.	Moreover, let $B$ be the operator in (\ref{B_esm_simpl_2}) and $\Up\in \Po(\O)$ be the sesquilinear form in (\ref{Up_esm_simpl_2}). The unitary operator $U_B$ in the polar decomposition of $M+B$ is defined by 
			$$
			(U_B f)(z)=\frac{r(z)\chi_{Z^c}(z)+\chi_Z(z)}{|r(z)\chi_{Z^c}(z)+\chi_Z(z)|}f(z), \qquad \forall f\in \H,
			$$
			where $Z:=\{z\in \C: |r(z)|\leq 1\}$, $\chi_Z$ and $\chi_{Z^c}$ are the characteristic functions of $Z$ and of its complement $Z^c$, respectively. It is easy to see that $U_B \D=\D$. Again we get that $\O$ is hyper-solvable.
		\end{esm}
		
		The following conditions are only sufficient to ensure that a solvable sesquilinear form is hyper-solvable.
		
		\begin{lem}
			Let $\O$ be a solvable sesquilinear form on $\D$ with respect to an inner product and with associated operator $T$. Let 
			$$
			\O(\xi,\eta)=\pin{QH\xi}{H\eta}, \qquad \forall \xi,\eta\in \D,
			$$			
			be a Radon-Nikodym-like representation of $\O$. Suppose, moreover, that $\Up\in \Po(\O)$, $B$ is the operator associated to $\Up$ and $Q_B:=Q+H^{-1}BH^{-1}$.
			If one of the following statements holds, then $\O$ is hyper-solvable.
			\begin{enumerate}
				\item $Q_B\D= \D$;
				\item $Q_B$ is positive;
				\item $T+B$ is semibounded.
			\end{enumerate}
		\end{lem}
		\begin{proof}
			Assume the first condition. We prove that $D(T)=D(H^2)$. Indeed, $\xi\in D(H^2)\sub D(H)=\D$ if and only if $H\xi\in \D$ if and only if $Q_B H\xi\in \D$, if and only if $\xi\in D(T+B)$, by the definition of domain of $T+B$.
			By Lemma \ref{P1,2^alfa} we have $D(|T|^\mez)=D(|T+B|^\mez)=D(H)=\D$.\\ For the conditions {\it 2.} and {\it 3.}, the statement follows by \cite[Lemma 3.6]{GKMV}.				
		\end{proof}

The next lemma is a consequence of Theorem \ref{cns_D=}. It will be used in the proof of Theorem \ref{2_repr_th_2} below, but it is interesting in itself.
		
		\begin{lem}
			\label{cor_D=}
			Let $\O$ be a solvable sesquilinear form on $\D$ with respect to an inner product with associated operator $T$. The following statements are equivalent.
			\begin{enumerate}
				\item $\D=D(|T|^\mez)$, i.e. $\O$ is hyper-solvable;
				\item $\D\subseteq D(|T|^\mez)\cap D(|T^*|^\mez)$;
				\item $\D\supseteq D(|T|^\mez)\cup D(|T^*|^\mez)$;
				\item $\D=D(|T^*|^\mez)$, i.e. $\O^*$ is hyper-solvable.
			\end{enumerate}
		\end{lem}
	\begin{proof}
		Let $\Up\in \Po(\O)$ and $B$ be the bounded operator associated to $\Up$. Then $\Up^*\in \Po(\O^*)$ by \cite[Theorem 4.11]{RC_CT}. Moreover, let $T+B=U_B|T+B|$ be the polar decomposition of $T+B$, then $T^*+B^*=U_B^*|T^*+B^*|$. \\
		Hence, taking into account the following simple equivalences,
		$U_B\D=\D$ if and only if, $U_B\D\subseteq \D$ and $U_B^*\D\subseteq \D$, if and only if, $U_B\D\supseteq \D$ and $U_B^*\D\supseteq \D$,  if and only if, $U_B^*\D=\D$ and we conclude applying Lemmas \ref{cns_D<} and \ref{cns_D>}.
	\end{proof}

	From this lemma  we obtain a counterpart of Theorem \ref{2_repr_th_1}.
	
	\begin{cor}
		Let $\O$ be a hyper-solvable sesquilinear form on $\D$ and  $T$ be the operator associated to $\O$. Suppose, moreover, that $\Up\in \Po(\O)$, and $B$ is the operator associated to $\Up$. \\
		Then, there exists a unique operator $Y\in \B(\H)$ such that 
		$$
		\O(\xi,\eta)=\pin{Y|T^*+B^*|^\mez\xi}{|T^*+B^*|^\mez\eta}, \qquad \forall \xi,\eta \in \D.
		$$
	\end{cor}

		If $\O$ is a solvable form on $\D$ represented by a normal operator $T$, then an inclusion between the subspace $\D$ and $D(|T|^\mez)$ is sufficient to prove that $\O$ is hyper-solvable. Therefore, the following corollary extends Proposition 2.5 of \cite{FHdeSW'} and Theorem 3.2 of \cite{GKMV}.
	
	\begin{cor}
		If $\O$ is a solvable sesquilinear form on $\D$ with respect to an inner product and with normal associated operator $T$ (in particular, if $\O$ is symmetric), then the following statements are equivalent.
		\begin{enumerate}
			\item $\D=D(|T|^\mez)$, i.e. $\O$ is hyper-solvable;
			\item $\D\sub D(|T|^\mez)$;
			\item $\D\supseteq D(|T|^\mez)$.
		\end{enumerate}
	\end{cor}

We now give the main result of this paper which generalizes to solvable sesquilinear forms Kato's second representation theorem \cite[Theorem VI.2.23]{Kato}.
		
		\begin{teor}
			\label{2_repr_th_2}
			Let $\O$ be a hyper-solvable sesquilinear form on $\D$ with respect to a norm $\noo$ and with associated operator $T$. Then 
			$$
			\O(\xi,\eta)=\pin{U|T|^\mez \xi}{|T^*|^\mez \eta}, \qquad \forall \xi,\eta \in \D,
			$$
			$$
			\O(\xi,\eta)=\pin{|T^*|^\mez U\xi}{|T^*|^\mez \eta}, \qquad \forall \xi,\eta \in \D,
			$$
			where $T=U|T|=|T^*|U$ is the polar decomposition of $T$, and $\noo$ is equivalent to the graph norms of $|T|^\mez$ and of $|T^*|^\mez$. 
		\end{teor}
		\begin{proof}
			By Corollary \ref{cor_hyper_inner} and Lemma \ref{cor_D=}, $\D=D(|T|^\mez)=D(|T^*|^\mez)$. Consider the sesquilinear form $\O'$ defined as follows
			$$
			\O'(\xi,\eta)=\pin{U|T|^\mez \xi}{|T^*|^\mez \eta}, \qquad  \xi,\eta \in \D.
			$$
			By the equality (see \cite[Sect. 3]{McIntosh85})
			\begin{equation}
			\label{UT^mez}
			|T^*|^\mez U=U|T|^\mez,
			\end{equation}
			we get
			\begin{equation}
			\label{T=T^*^mezUT^mez}
			T=|T^*|U=|T^*|^\mez|T^*|^\mez U =|T^*|^\mez U |T|^\mez.
			\end{equation}
			Hence,
			\begin{eqnarray*}
			\O(\xi,\eta) &=& \pin{T\xi}{\eta} \\
			&=& \pin{|T^*|^\mez U |T|^\mez \xi}{\eta} \\
			&=& \pin{U|T|^\mez \xi}{|T^*|^\mez \eta} \\
			&=& \O'(\xi,\eta),
			\end{eqnarray*}
			for all $\xi\in D(T), \eta \in \D$. The forms $\O$ and $\O'$ are q-closed with respect to the same norm $\noo$ (which is equivalent to the graph norms of $|T|^\mez$ and of $|T^*|^\mez$ by Lemma \ref{P1,2^alfa} and Proposition \ref{pro_equiv_norm}) and, in particular, they are bounded in $\D[\noo]$. Since $D(T)$ is dense in $\D[\noo]$, $\O$ and $\O'$ coincide by continuity on the whole $\D$.	The second equality follows by (\ref{UT^mez}).
		\end{proof}

		\begin{cor}
			Let $\O$ be a hyper-solvable sesquilinear form on $\D$ with associated operator $T$. Then 
			$$
			\O^*(\xi,\eta)=\pin{U^*|T^*|^\mez \xi}{|T|^\mez \eta}, \qquad \forall \xi,\eta \in \D,
			$$
			$$
			\O^*(\xi,\eta)=\pin{|T|^\mez U^*\xi}{|T|^\mez \eta}, \qquad \forall \xi,\eta \in \D,
			$$
			where $T=U|T|=|T^*|U$ is the polar decomposition of $T$.
		\end{cor}

  		As a consequence, we get the following statements which generalizes Theorem 4.2 of \cite{FHdeS}, Theorem 2.10 of \cite{GKMV} and Theorem 3.1 of \cite{Schmitz}. 
		
		\begin{cor}
			Let $\O$ be a hyper-solvable sesquilinear form on $\D$ with respect to a norm $\noo$ and with associated operator $T$. If $T$ is normal (in particular, if $\O$ is symmetric), then 
			$$
			\O(\xi,\eta)=\pin{U|T|^\mez \xi}{|T|^\mez \eta}=\pin{|T|^\mez U \xi}{|T|^\mez \eta}, \qquad \forall \xi,\eta \in \D,
			$$
			where $T=U|T|$ is the polar decomposition of $T$.
		\end{cor}

\begin{osr}
	Symmetric hyper-solvable sesquilinear forms have also a formulation given by the spectral theorem for self-adjoint operators (see \cite[Sect. 10.2]{Schm}).
\end{osr}

\begin{esm}
	One can easily check that the hyper-solvable sesquilinear form $\O_{\ull{\alpha}}$ of Example \ref{esm_simpl_1} satisfies 
	$$
	\Oa(\{\xi_n\},\{\eta_n\})=\pin{U|\Ma|^\mez \{\xi_n\}}{|\Ma|^\mez \{\eta_n\}}=\pin{|\Ma|^\mez U \{\xi_n\}}{|\Ma|^\mez \{\eta_n\}},  
	$$
	for all $\{\xi_n\},\{\eta_n\}\in \D$, where $\Ma=U|\Ma|$ is the polar decomposition of the operator $\Ma$ associated to $\O_{\ull{\alpha}}$. More precisely, for all $\{\xi_n\}\in \D$
	$$
	|\Ma|^\mez \{\xi_n\}=\{|\alpha_n|^\mez \xi_n\} \qquad \text{and} \qquad
	U\{\xi_n\}=\{\chi_n\},
	$$
	where $\chi_n=\frac{\alpha_n}{|\alpha_n|}\xi_n$ if $\alpha_n\neq 0$, $\chi_n=0$ otherwise.
\end{esm}

The sesquilinear form of Example \ref{esm_simpl_2} admits similar Kato's second type representations.

\section{Correspondence between solvable sesquilinear  forms and operators}
\label{sec:corr}

		With the help of the results of the previous section we can give a partial answer to the following question: which properties must an operator $T$ have to ensure the existence of a solvable sesquilinear form that is represented by $T$?

		\begin{teor}
			\label{th_inverse}
			Let $T$ be a densely defined, closed operator which satisfies
			\begin{enumerate}
				\item[(a)] $D(|T|^\mez)=D(|T^*|^\mez)$;
				\item[(b)] there exists $B\in \B(\H)$ such that $0\in \rho(T+B)$.
			\end{enumerate}
			Then, there exists a unique hyper-solvable sesquilinear form $\O$ with associated operator $T$.
		\end{teor}
		
		\begin{proof}
			Let $T+B=U_B|T+B|=|T^*+B^*|U_B$ be the polar decomposition of $T+B$. Setting $\D:=D(|T|^\mez)=D(|T^*|^\mez)$, then 
			$$\D=D(|T+B|^\mez)=D(|T^*+B^*|^\mez).$$
			We define a sesquilinear form on $\D$ by putting
			\begin{equation*}
			\O_B(\xi,\eta)=\pin{U_B|T+B|^\mez\xi}{|T^*+B^*|^\mez\eta}, \qquad  \xi,\eta\in \D.
			\end{equation*}
			The hypothesis ensures that $0\in \rho(|T+B|^\mez)\cap \rho(|T^*+B^*|^\mez)$ and that $U_B$ is unitary, hence by Theorem \ref{th_H1_H2}, $\O_B$ is a solvable sesquilinear form with $0\in \Po(\O_B)$. The operator associated to $\O_B$ is $T'=|T^*+B^*|^\mez U_B |T+B|^\mez$ on the natural domain
			$D(T')=\{\xi\in \D:U_B|T+B|^\mez\xi \in \D\}.$
			But, by (\ref{T=T^*^mezUT^mez}), $T'=|T^*+B^*| U_B=T+B$.\\
			The statement is proved noting that $\O:=\O_B-\Up$ (where $\Up(\cdot,\cdot)=\pin{B\cdot}{\cdot}$) is an hyper-solvable sesquilinear form with associated operator $T$.\\
			If $\O'$ is a hyper-solvable sesquilinear form represented by $T$, then  $\O=\O'$ by Theorem \ref{2_repr_th_2}. 
		\end{proof}
		
		\begin{osr}
			Condition (a) above is satisfied if $D(T)=D(T^*)$ and, in particular, if $T$ is normal.
		\end{osr}

			All the conditions on the operator $T$ listed in the statement of the previous theorem are also necessary for the existence of a hyper-solvable sesquilinear form represented by $T$. Indeed, an operator $S$ associated to a solvable form $\O$ is densely defined, closed and such that $0\in \rho(S+B)$, for some $B\in \B(\H)$, by Theorem \ref{th_rapp_risol}. Consequently, $0\in \rho(S^*+B^*)$. Finally, from Lemma \ref{cor_D=}, if $\O$ is hyper-solvable, then $D(|S|^\mez)=D(|S^*|^\mez)$. Therefore, we can formulate the following theorem.

		\begin{teor}
			\label{corr_gen}
			The map $\O\to T$ defined by Theorem \ref{th_rapp_risol} establishes a one-to-one correspondence between all hyper-solvable sesquilinear forms $\O$ and all densely defined closed operators $T$ such that $D(|T|^\mez)=D(|T^*|^\mez)$ and there exists $B\in \B(\H)$ satisfying $0\in \rho(T+B)$.
		\end{teor}

		In the case of symmetric forms we obtain the next correspondence.		

		\begin{cor}
			\label{corr_symm}
			The map $\O\to T$ defined by Theorem \ref{th_rapp_risol} establishes a one-to-one correspondence between all symmetric hyper-solvable sesquilinear forms and all self-adjoint operators.
		\end{cor}
		
		We want to emphasize that infinitely many solvable sesquilinear forms might be represented by the same operator (see Proposition 4.2 of \cite{GKMV}).
		
		Theorem \ref{th_inverse} generalizes Proposition 5.1 of \cite{FHdeS}.
		While using Corollary 7.6 of \cite{RC_CT} we get Theorem 5.2 of \cite{FHdeS} as a special case of Theorem \ref{corr_gen} and Corollary \ref{corr_symm}.

\section*{Acknowledgments}
	
The author is grateful to Prof. C. Trapani for reading the paper and for his comments. This work has been done in the framework of the project ''Problemi spettrali e di rappresentazione in quasi *-algebre di operatori'', INDAM-GNAMPA 2017.

\vspace*{0.5cm}
\textsc{Rosario Corso, Dipartimento di Matematica e Informatica, Università degli Studi di Palermo, I-90123 Palermo, Italy}

{\it E-mail address}: {\bf rosario.corso@studium.unict.it}
	
\end{document}